\documentclass[11pt, a4paper]{article}

\title{Inner-iteration preconditioning\\with a symmetric splitting matrix\\for rank-deficient least squares problems}

\author{Keiichi Morikuni\footnote{\texttt{morikuni@cs.cas.cz}} \footnote{Institute of Computer Science, The Czech Academy of Sciences. Current affiliation: Faculty of Engineering, Information and Systems, University of Tsukuba (\texttt{morikuni@cs.tsukuba.ac.jp}).}}
\date{\today}

\usepackage{amsmath, amsfonts, amssymb, amsthm}
\usepackage[all, warning]{onlyamsmath}
\RequirePackage[l2tabu, orthodox]{nag} 
\usepackage[bookmarks=true, bookmarksnumbered=true, bookmarkstype=toc,
 pdftitle={Inner-iteration preconditioning with a symmetric splitting matrix for rank-deficient least squares problems},
 pdfauthor={Keiichi Morikuni},%
 pdfkeywords={Rank-deficient least squares problems Preconditioning Krylov subspace methods Symmetric singular linear systems},%
 colorlinks={true},%
 urlcolor={red},
 setpagesize={false},
 pdfcreator={},
 pdfdisplaydoctitle={true}
]{hyperref}
\usepackage[T1]{fontenc}
\usepackage{textcomp}
\usepackage{lmodern}
\usepackage{algorithm, algorithmic}
\usepackage{cite} 

\theoremstyle{plain} 
\newtheorem{theorem}{Theorem}[section]
\newtheorem{lemma}[theorem]{Lemma}

\newtheorem{remark}[theorem]{Remark}

\numberwithin{equation}{section}
\makeatletter
	
	\@addtoreset{equation}{section}
\makeatother

\newfont{\bg}{cmr9 scaled\magstep4}
\newcommand{\bigzerol}{\smash{\lower1.0ex\hbox{\bg 0}}}

\DeclareMathOperator{\diag}{diag}

\DeclareMathOperator{\I}{I}

\begin{document}
\maketitle

\begin{abstract}
Stationary iterative methods with a symmetric splitting matrix are performed as inner-iteration preconditioning for Krylov subspace methods.
We give conditions such that the inner-iteration preconditioning matrix is definite, and show that conjugate gradient (CG) method preconditioned by the inner iterations determines a solution of symmetric and positive semidefinite linear systems, and the minimal residual (MINRES) method preconditioned by the inner iterations determines a solution of symmetric linear systems including the singular case.
These results are applied to the CG and MINRES-type methods such as the CGLS, LSMR, and CGNE methods preconditioned by inner iterations, and thus justify using these methods for solving least squares and minimum-norm solution problems whose coefficient matrices are not necessarily of full rank.
Thus, we complement the convergence theories of these methods presented in [K.\ Morikuni and K.\ Hayami, SIAM J.\ Matrix Appl.\ Anal., 34 (2013), pp.~1--22], [K.\ Morikuni and K.\ Hayami, SIAM J.\ Matrix Appl.\ Anal., 36 (2015), pp.~225--250], and give bounds for these methods.

\smallskip
\noindent \textbf{Keywords:} Rank-deficient least squares problems, Preconditioning, Krylov subspace methods, Symmetric singular linear systems.

\noindent \textbf{AMS subject classifications:} 65F08, 65F10, 65F20, 65F50.
\end{abstract}

\section{Introduction.}
First, consider solving symmetric linear systems of equations
\begin{align}
	\mathbf{A} \mathbf{x} = \mathbf{b},
	\label{eq:LS}
\end{align}
where $\mathbf{A} = \mathbf{A}^\mathsf{T} \in \mathbb{R}^{n \times n}$ may be singular and $\mathbf{b}$ is in the range space of $\mathbf{A}$, $\mathcal{R}(\mathbf{A})$.

In the symmetric and positive definite (SPD) case, i.e., $\mathbf{v}^\mathsf{T} \! \mathbf{A} \mathbf{v} > 0$ for all $\mathbf{v} \not = \mathbf{0}$, the conjugate gradient (CG) method \cite{HestenesStiefel1952} has been used.
In the symmetric and positive semidefinite (SPSD) case, i.e., $\mathbf{v}^\mathsf{T} \! \mathbf{A} \mathbf{v} \geq 0$ for all $\mathbf{v} \in \mathbb{R}^n$ and for all $\mathbf{b} \in \mathcal{R}(\mathbf{A})$, CG with the initial iterate $\mathbf{x}_0 \in \mathbb{R}^n$ determines the $k$th iterate $\mathbf{x}_k \in \mathbf{x}_0 + \mathcal{K}_k (\mathbf{A}, \mathbf{r}_0)$ that minimizes the $\mathbf{A}$-seminorm $\| \mathbf{e}_k \|_{\mathbf{A}} = \| \mathbf{x}_k - \mathbf{x}_* \|_{\mathbf{A}} $, where $\mathbf{r}_0 = \mathbf{b} - \mathbf{A} \mathbf{x}_0$, $\mathcal{K}_k (\mathbf{A}, \mathbf{r}_0) = \mathrm{span} \lbrace \mathbf{r}_0, \mathbf{A} \mathbf{r}_0, \dots, \mathbf{A}^{k-1} \mathbf{r}_0 \rbrace$ is the Krylov subspace of order $k$, $\| \mathbf{e} \|_{\mathbf{A}} = \sqrt{\mathbf{e}^\mathsf{T} \mathbf{A} \mathbf{e} }$ is the seminorm associated with $\mathbf{A}$ SPSD,
\begin{align}
	\mathbf{x}_* = \mathbf{A}^\dag \mathbf{b} + (\mathbf{I} - \mathbf{A}^\dag \! \mathbf{A}) \mathbf{x}_0,
	\label{eq:Ksol}
\end{align}
and $\mathbf{I}$ is the identity matrix (see \cite[Theorem~3.2]{WeiWu2000}).
Here, $\mathbf{A}^\dag$ is the pseudoinverse of $\mathbf{A}$.
CG determines the solution $\mathbf{x}_*$ of \eqref{eq:LS} for all $\mathbf{b} \in \mathcal{R}(\mathbf{A})$ and for all $\mathbf{x}_0 \in \mathbb{R}^n$, and determines the minimum-norm solution $\mathbf{A}^\dag \mathbf{b}$ of \eqref{eq:LS} for all $\mathbf{b} \in \mathcal{R}(\mathbf{A})$ and for all $\mathbf{x}_0 \in \mathcal{R}(\mathbf{A})$ \cite{KammererNashed1972SINUM, HayamiYin2008}.
An error bound of CG is given by $\| \mathbf{e}_k \|_{\mathbf{A}} \leq 2 [ (\sqrt{\kappa_2 (\mathbf{A})} - 1) / (\sqrt{\kappa_2 (\mathbf{A})} + 1)]^k \| \mathbf{e}_0 \|_{\mathbf{A}}$ \cite{Kaniel1966}, where $\kappa_2 (\mathbf{A}) = \| \mathbf{A} \|_2 \| \mathbf{A}^{\dag} \|_2$.
Hence, the convergence is expected to be fast as $\kappa_2 (\mathbf{A})$ is small.

In the indefinite case, the minimal residual (MINRES) method \cite{PaigeSaunders1975} has been used.
For $\mathbf{b} \in \mathbb{R}^n$, MINRES with $\mathbf{x}_0 \in \mathbb{R}^n$ determines the $k$th iterate $\mathbf{x}_k \in \mathbf{x}_0 + \mathcal{K}_k (\mathbf{A}, \mathbf{r}_0)$ that minimizes $\| \mathbf{r}_k \|_2$.
MINRES determines a solution of least squares problems $\min_{\mathbf{x} \in \mathbb{R}^n} \| \mathbf{b} - \mathbf{A} \mathbf{x} \|_2$ for all $\mathbf{b} \in \mathbb{R}^n$ and for all $\mathbf{x}_0 \in \mathbb{R}^n$, determines the solution of the form \eqref{eq:Ksol} for all $\mathbf{b} \in \mathcal{R}(\mathbf{A})$ and for all $\mathbf{x}_0 \in \mathbb{R}^n$, and determines the minimum-norm solution of \eqref{eq:LS} for all $\mathbf{b} \in \mathcal{R}(\mathbf{A})$ and for all $\mathbf{x}_0 \in \mathcal{R}(\mathbf{A})$.
These arguments are given by specializing the convergence analysis of the generalized minimal residual (GMRES) method \cite{BrownWalker1997, HayamiSugihara2011} to the symmetric case.
Similar to CG, a residual bound of MINRES is given by $\| \mathbf{r}_k \|_2 \leq \varepsilon^k \| \mathbf{r}_0 \|_2$ with $\varepsilon^k = \min_{\genfrac{}{}{0pt}{2}{p \in \mathbb{P}_k}{p(0) = 1}} \max_{\lambda \in \sigma (\mathbf{A})} | p(\lambda) |$, where $\mathbb{P}_k$ is the set of all polynomials of degree not exceeding $k$ and $\sigma (\mathbf{A})$ is the spectrum of $\mathbf{A}$ \cite[Theorem~1]{Bai2000}.
See \cite{FoxHuskeyWilkinson1948, HestenesStiefel1952, Stiefel1955, PaigeSaunders1975, Vinsome1976, YoungJea1980, EisenstatElmanSchultz1983} for other Krylov subspace methods for symmetric linear systems.

For accelerating the convergence of CG and MINRES, consider using preconditioning.
See \cite{Kaasschieter1988} for the preconditioned CG method in the singular case.
Several steps of stationary iterative methods serve as preconditioning for Krylov subspace methods, which may be considered as inner iterations \cite{MorikuniHayami2013}. 
We consider using stationary iterative methods with a symmetric splitting matrix as inner-iteration preconditioning for CG and MINRES.
To show that these methods determine a solution of symmetric and indefinite linear systems including the singular case, we give conditions such that the inner-iteration preconditioning matrix is SPD and give convergene bounds for these methods.
The conditions are satisfied by the Richardson, Jacobi overrelaxation (JOR), and symmetric successive overrelaxation (SSOR) methods \cite{Richardson1911, Jacobi1845, Sheldon1955}.
Thus, we extend  the theories in the SPD case \cite{DuboisGreenbaumRodrigue1979, Adams1985} to a general symmetric case.
Also, inner-iteration preconditioning is regarded as an extension of the splitting preconditioning \cite[Section 10.2]{Saad2003}.
An extension to symmetric positive semidefinite systems was considered in \cite{SugiharaHayami2016TJSIAM}.

These methods can be used to determine a solution of the normal equations
\begin{align}
	A^\mathsf{T} \! A \boldsymbol{x} = A^\mathsf{T} \boldsymbol{b},
	\label{eq:normeq}
\end{align}
equivalently least squares problems
\begin{align}
	\min_{\boldsymbol{x} \in \mathbb{R}^{n}} \| \boldsymbol{b} - A \boldsymbol{x} \|_2,
	\label{eq:LSprob}
\end{align}
where $A \in \mathbb{R}^{m \times n}$ is not necessarily of full rank and $\boldsymbol{b} \in \mathbb{R}^m$ is not necessarily in $\mathcal{R}(A)$.
For solving \eqref{eq:normeq}, we can use efficient implementations of CG and MINRES such as the CGLS, LSQR, and LSMR methods \cite{HestenesStiefel1952, PaigeSaunders1982a, FongSaunders2011}.
For solving \eqref{eq:LSprob}, the (preconditioned) CGLS and LSQR methods have been used, which both are mathematically equivalent to (preconditioned) CG applied to \eqref{eq:normeq}.
Another option for solving \eqref{eq:LSprob} is to use the (preconditioned) LSMR method \cite{FongSaunders2011}, which is mathematically equivalent to (preconditioned) MINRES applied to \eqref{eq:normeq}.

On the other hand, consider solving minimum-norm solution problems 
\begin{align}
	 \min \| \boldsymbol{x} \|_2, \quad \mbox{subject to} \quad A \boldsymbol{x} = \boldsymbol{b}, \quad \boldsymbol{b} \in \mathcal{R}(A).
	 \label{eq:minsolLS}
\end{align}
The solution of \eqref{eq:minsolLS} is the pseudo-inverse solution of $A \boldsymbol{x} = \boldsymbol{b}$, $\boldsymbol{b} \in \mathcal{R}(A)$.
The problem \eqref{eq:minsolLS} is equivalent to the normal equations of the second kind 
\begin{align}
	\boldsymbol{x} = A^\mathsf{T} \boldsymbol{u}, \quad \mbox{subject to} \quad A A^\mathsf{T} \boldsymbol{u} = \boldsymbol{b}, \quad \boldsymbol{b} \in \mathcal{R}(A).
	\label{eq:normaleq2}
\end{align}
Note that the constraint of \eqref{eq:normaleq2} is an SPSD linear system.

For solving \eqref{eq:minsolLS}, the (preconditioned) CGNE method \cite{Craig1955} has been used, which is mathematically equivalent to (preconditioned) CG applied to the constraint of \eqref{eq:normaleq2}.
Another option for solving \eqref{eq:minsolLS} is to use the (preconditioned) MRNE method \cite{MorikuniHayami2015}, which is mathematically equivalent to (preconditioned) MINRES applied to the constraint of \eqref{eq:normaleq2}.
We apply the above mentioned result for symmetric linear systems to CGLS, LSQR, LSMR, CGNE, and MRNE preconditioned by inner iterations, and thus justify using these methods particularly for rank-deficient least squares problems and minimum-norm solution problems.

In this paper, we complement the theory for the inner-iteration preconditioning for the CG and MINRES-type methods including the rank-deficient case.
These methods have an advantage concerning memory requirement compared to the right- and left-preconditioned GMRES methods for least squares problems \cite{HayamiYinIto2010, MorikuniHayami2013, MorikuniHayami2015}.
CGLS and CGNE preconditioned by one step of SSOR-type methods were proposed in \cite{BjorckElfving1979}.
These methods were generalized to multistep versions in \cite{MorikuniHayami2013, MorikuniHayami2015}.

The rest of the paper is organized as follows.
In Section~\ref{sec:symmLS}, we give conditions such that CG and MINRES preconditioned by inner iterations determine a solution of linear systems, give bounds of these methods, and derive conditions for specific stationary iterative methods that satisfy the conditions.
In Sections~\ref{sec:aplLS} and \ref{sec:aplminsol}, we apply these results to CGLS, LSQR, and LSMR preconditioned by inner iterations for solving least squares problems and CGNE and MRNE preconditioned by inner iterations for solving minimum-norm solution problems, respectively.
In Section~\ref{sec:conc}, we conclude the paper.

\section{Preconditioning for symmetric linear systems.} \label{sec:symmLS}
Consider solving symmetric linear systems \eqref{eq:LS}.
Let $\mathbf{P} = \mathbf{P}^\mathsf{T} \in \mathbb{R}^{n \times n}$ be a preconditioning matrix for \eqref{eq:LS}.
If $\mathbf{P}$ is SPD, then the linear system \eqref{eq:LS} is equivalent to the preconditioned one $\mathbf{P}^{-1} \mathbf{A} \mathbf{x} = \mathbf{P}^{-1} \mathbf{b}$, or
\begin{align}
	\mathbf{P}^{-\frac{1}{2}} \mathbf{A} \mathbf{P}^{-\frac{1}{2}} \mathbf{y} = \mathbf{P}^{-\frac{1}{2}} \mathbf{b}, \quad \mathbf{x} = \mathbf{P}^{-\frac{1}{2}} \mathbf{y}
	\label{eq:PLS}
\end{align}
for all $\mathbf{b} \in \mathcal{R}(\mathbf{A})$, where $\mathbf{P}^{\frac{1}{2}}$ is the square root of $\mathbf{P}$. 
If $\hat{\mathbf{A}} = \mathbf{P}^{-\frac{1}{2}} \mathbf{A} \mathbf{P}^{-\frac{1}{2}}$, $\hat{\mathbf{x}} = \mathbf{P}^{\frac{1}{2}} \mathbf{x}$, and $\hat{\mathbf{b}} = \mathbf{P}^{-\frac{1}{2}} \mathbf{b}$, \eqref{eq:PLS} becomes $\hat{\mathbf{A}} \hat{\mathbf{x}} = \hat{\mathbf{b}}$.

For $\mathbf{b} \in \mathcal{R}(\mathbf{A})$, CG applied to \eqref{eq:PLS} (PCG) determines $\mathbf{x}_k \in \mathbf{x}_0 + \mathcal{K}_k (\mathbf{P}^{-1} \mathbf{A}, \mathbf{P}^{-1} \mathbf{r}_0)$ that minimizes $\| \hat{\mathbf{x}}_k - \hat{\mathbf{x}}_* \|_{\mathbf{A}}$, equivalently CG applied to $\mathbf{P}^{-1} \mathbf{A} \mathbf{x} = \mathbf{P}^{-1} \mathbf{b}$ with the $\mathbf{P}$-inner product does this, where
\begin{align}
	\hat{\mathbf{x}}_* = \mathbf{P}^{-\frac{1}{2}} \hat{\mathbf{A}}^\dag \hat{\mathbf{b}} + \mathbf{P}^{-\frac{1}{2}} (\mathbf{I} - \hat{\mathbf{A}}^\dag \hat{\mathbf{A}}) \hat{\mathbf{x}}_0
	\label{eq:PKsol}
\end{align}
(cf.~\eqref{eq:Ksol}).
For $\mathbf{b} \in \mathcal{R}(\mathbf{A})$, MINRES applied to \eqref{eq:PLS} (PMR) determines $\mathbf{x}_k \in \break \mathbf{x}_0 + \mathcal{K}_k (\mathbf{P}^{-1} \mathbf{A}, \mathbf{P}^{-1} \mathbf{r}_0)$ that minimizes $\| \hat{\mathbf{r}}_k \|_2$, equivalently MINRES applied to $\mathbf{P}^{-1} \mathbf{A} \mathbf{x} = \mathbf{P}^{-1} \mathbf{b}$ with the $\mathbf{P}$-inner product does this, where $\hat{\mathbf{r}}_k = \hat{\mathbf{b}} - \hat{\mathbf{A}} \hat{\mathbf{x}}_k$.
On the other hand, if $\mathbf{P}$ is symmetric and negative definite (SND), i.e., $\mathbf{v}^\mathsf{T} \! \mathbf{A} \mathbf{v} < 0$ for all $\mathbf{v} \not = \mathbf{0}$, then the linear system $\mathbf{A} \mathbf{x} = \mathbf{b}$ is equivalent to the preconditioned one $(-\mathbf{P})^{-1} \mathbf{A} \mathbf{x} = (-\mathbf{P})^{-1} \mathbf{b}$, or $(-\mathbf{P})^{-\frac{1}{2}} \mathbf{A} (-\mathbf{P})^{-\frac{1}{2}} \mathbf{y} = (-\mathbf{P})^{-\frac{1}{2}} \mathbf{b}$, $\mathbf{x} = (-\mathbf{P})^{-\frac{1}{2}} \mathbf{y}$ for all $\mathbf{b} \in \mathcal{R}(\mathbf{A})$.
Without loss of generality, we restrict ourselves to the case where the preconditioning matrix is SPD for simplicity hereafter.
Even when the preconditioning matrix is SND, the arguments below hold by changing the sign.
Thus, we obtain the following.

\begin{lemma}{(\cite{Kaasschieter1988})} \label{lm:PCG}
Assume that $\mathbf{A}$ is SPSD and $\mathbf{P}$ is SPD.
Then, PCG determines a solution of $\mathbf{A} \mathbf{x} = \mathbf{b}$ for all $\mathbf{b} \in \mathcal{R}(\mathbf{A})$ and for all $\mathbf{x}_0 \in \mathbb{R}^n$.
The solution is of the form \eqref{eq:PKsol}.
\end{lemma}

\begin{lemma}{(\cite{PaigeSaunders1975,BrownWalker1997,HayamiSugihara2011})} \label{lm:PMR}
Assume thet $\mathbf{A} = \mathbf{A}^\mathsf{T}$ and $\mathbf{P}$ is SPD.
Then, PMR determines a solution of $\mathbf{A} \mathbf{x} = \mathbf{b}$ for all $\mathbf{b} \in \mathcal{R}(\mathbf{A})$ and for all $\mathbf{x}_0 \in \mathbb{R}^n$.
The solution is of the form \eqref{eq:PKsol}.
\end{lemma}

We note that PCG and PMR do not necessarily determine the minimum-norm solution $\mathbf{A}^\dag \mathbf{b}$.
Under the assumptions in Lemmas~\ref{lm:PCG} and \ref{lm:PMR}, PCG and PMR respectively determine the weighted minimum-norm solution $\arg\min \| \mathbf{P}^{\frac{1}{2}} \boldsymbol{x} \|_2$, subject to $\mathbf{A} \mathbf{x} = \mathbf{b}$ for all $\mathbf{b} \in \mathcal{R}(\mathbf{A})$ and for all $\mathbf{x}_0 \in \mathcal{R}(\mathbf{P}^{-1} \mathbf{A})$.

\subsection{Inner-iterations preconditioned methods.} \label{sec:iCG}
Consider using $\ell$ steps of a stationary iterative method as inner-iteration preconditioning for CG for SPSD linear systems.
Let $\mathbf{C}^{(\ell)}$ be the preconditioning matrix of $\ell$ inner iterations.
An algorithm of this method is given as follows \cite{Axelsson1976} (see \cite{EisenstatOrtegaVaughan1990} for efficient implementations).

\begin{algorithm}
	\caption{CG method preconditioned by $\ell$ inner iterations.}
		\label{alg:CGin}
		\begin{algorithmic}[1]
		\STATE Let $\mathbf{x}_0 \in \mathbb{R}^n$ be the initial iterate and $\mathbf{r}_0 = \mathbf{b} - \mathbf{A} \mathbf{x}_0$.
		\STATE Apply $\ell$ steps of a stationary iterative method to $\mathbf{A} \mathbf{z} =\mathbf{r}_0$ to obtain $\mathbf{z}_0 = \mathbf{p}_0 = \mathbf{C}^{(\ell)} \mathbf{r}_0$. 
		\FOR{$k = 1, 2, \dots$ until convergence}
		\STATE $\alpha_k = {\mathbf{r}_k}^\mathsf{T} \mathbf{z}_k / {\mathbf{p}_k}^\mathsf{T} \mathbf{A} \mathbf{p}_k$, $\mathbf{x}_{k+1} = \mathbf{x}_k + \alpha_k \mathbf{p}_k$, $\mathbf{r}_{k+1} = \mathbf{r}_k - \alpha_k \mathbf{A} \mathbf{p}_k$
		\STATE Apply $\ell$ steps of a stationary iterative method to $\mathbf{A} \mathbf{z} =\mathbf{r}_{k+1}$ to obtain $\mathbf{z}_{k+1} \! = \! \mathbf{C}^{(\ell)}$.
		\STATE $\beta_k  =  {\mathbf{r}_{k+1}}^\mathsf{T} \mathbf{z}_{k+1} / {\mathbf{r}_k}^\mathsf{T} \mathbf{z}_k$, $\mathbf{p}_{k+1} = \mathbf{z}_{k+1} + \beta_k \mathbf{p}_k$
		\ENDFOR
	\end{algorithmic}
\end{algorithm}

We form the preconditioned matrix for CG with $\ell$ inner iterations. 
Consider the stationary iterative method applied to $\mathbf{A} \mathbf{z} = \mathbf{r}_k$ in lines 2 and 5.
We call $\mathbf{A} = \mathbf{M} - \mathbf{N}$ a splitting of $\mathbf{A}$ and assume that $\mathbf{M}$ is nonsingular.
Denote the iteration matrix by $\mathbf{H} = \mathbf{M}^{-1} \mathbf{N}$.
Assume that the initial iterate $\mathbf{z}^{(0)}$ is in the nullspace of $\mathbf{H}$, e.g., $\mathbf{z}^{(0)} = \mathbf{0}$.
Then, the $\ell$th iterate of the stationary iterative method is 
$\mathbf{z}^{(\ell)} = \mathbf{H} \mathbf{z}^{(\ell - 1)} + \mathbf{M}^{-1} \mathbf{r}_k = \sum_{i = 0}^{\ell - 1} \mathbf{H}^i \mathbf{M}^{-1} \mathbf{r}_k, \ell \in \mathbb{N}$.
Hence, the inner-iteration preconditioning matrix is $\mathbf{C}^{(\ell)} = \sum_{i = 0}^{\ell - 1} \mathbf{H}^i \mathbf{M}^{-1}$.
Therefore, the preconditioned matrix is $\mathbf{C}^{(\ell)} \mathbf{A} = \sum_{i = 0}^{\ell - 1} \mathbf{H}^i (\mathbf{I} - \mathbf{H}) = \mathbf{I} - \mathbf{H}^\ell$.
Conditions such that $\mathbf{C}^{(\ell)}$ is SPD will be given in Section~\ref{sec:definit}.
Under the conditions, one can set $\mathbf{P}^{-1} = \mathbf{C}^{(\ell)}$ and the assumptions in Lemma~\ref{lm:PCG} are satisfied.
The inner-iteration preconditioning can be considered as the polynomial preconditioning using the truncated Neuman series expansion of $\mathbf{M}^{-1} \! \mathbf{A}$ (see \cite{Cesari1937, DuboisGreenbaumRodrigue1979, EisenstatOrtegaVaughan1990, FaberJoubertKnillManteuffel1996, CerdanMarinMartinez2002} and references therein).

On the other hand, consider the case of MINRES.
Its algorithm is given as follows (cf.\ \cite{EisenstatOrtegaVaughan1990}).

\begin{algorithm}
\caption{MINRES method preconditioned by $\ell$ inner iterations.}
\label{alg:MRin}
\begin{algorithmic}[1]
	\STATE Let $\mathbf{x}_0 \in \mathbb{R}^n$ be the initial iterate, $\mathbf{w}_0 = \boldsymbol{0}$, and $\mathbf{w}_1 = \mathbf{b} - \mathbf{A} \mathbf{x}_0$.
	\STATE Apply $\ell$ steps of a stationary iterative method to $\mathbf{A} \mathbf{z} =\mathbf{w}_1$ to obtain $\mathbf{z}_1 = \mathbf{C}^{(\ell)} \mathbf{w}_1$.
	\STATE $\mathbf{s}_0 = \mathbf{s}_{-1} = \boldsymbol{0}$, $c_0 = -1$, $s_0 = 0$, $\beta_0 = 1$, $\gamma_1 = 0$, $\beta_1 = ({\mathbf{w}_1}^\mathsf{T} \mathbf{z}_1)^{\frac{1}{2}}$, $\varphi_0 = \xi_0 = \beta_1$
	\FOR{$k = 1, 2, \dots$ until convergence}
	\STATE $\mathbf{p}_k = \mathbf{A} \mathbf{z}_k$, $\alpha_k = \mathbf{z}_k^\mathsf{T} \mathbf{p}_k / \beta_k^2$, $\mathbf{w}_{k+1} = (1/\beta_k) \mathbf{p}_k - (\alpha_k / \beta_k) \mathbf{w}_k - (\beta_k / \beta_{k-1}) \mathbf{w}_{k-1}$
\STATE Apply $\ell$ steps of a stationary iterative method to $\mathbf{A} \mathbf{z} \!\!  = \!\mathbf{w}_{k+1}$ to obtain $\mathbf{z}_{k+1} \! \! = \! \mathbf{C}^{(\ell)} \mathbf{w}_{k+1}$.
	\STATE $\beta_{k+1} = (\mathbf{w}_{k+1}^\mathsf{T} \mathbf{z}_{k+1})^{\frac{1}{2}}$, $\zeta_k = c_{k-1} \gamma _k + s_{k-1} \alpha_k$, $\theta_k = s_{k-1} \gamma_k - c_{k-1} \alpha_k$, $\tau_{k+1} = s_{k-1} \beta_{k+1}$
	\STATE $\gamma_{k+1} = - c_{k-1} \beta_{k+1}$, $\eta_k = ({\theta_k}^2 + {\beta_{k+1}}^2)^{\frac{1}{2}}$, $c_k = \theta_k / \eta_k$, $s_k = \beta_{k+1} / \eta_k$
	\STATE $\varphi_k = c_k \xi_{k-1}$, $\xi_k = s_k \xi_{k-1}$, $\mathbf{s}_k = ((1/\beta_k) \mathbf{z}_k - \zeta_k \mathbf{s}_{k-1} - \tau_k \mathbf{s}_{k-2}) / \eta_k$, $\mathbf{x}_k = \mathbf{x}_{k-1} + \varphi_k \mathbf{s}_k$
	\ENDFOR
\end{algorithmic}
\end{algorithm}

The stationary iterative method applied to $\mathbf{A} \mathbf{z} = \mathbf{w}_k$ in lines 2 and 6, respectively, gives the same preconditioning and preconditioned matrices as those in Algorithm~\ref{alg:CGin}.

\subsection{Definiteness of inner-iteration preconditioning matrices.} \label{sec:definit}
In order to examine the definiteness of the inner-iteration preconditioning matrix $\mathbf{C}^{(\ell)}$, we extend \cite[Lemma~1 and Theorem~1]{Adams1985} to the general symmetric case $\mathbf{A} = \mathbf{A}^\mathsf{T}$.
We denote $\mathbf{A} \sim \mathbf{B}$ if $\mathbf{A}$ and $\mathbf{B}$ are similar and $\mathbf{A} \equiv \mathbf{B}$ if $\mathbf{A}$ and $\mathbf{B}$ are congruent.

\begin{lemma}\label{lm:ABC}
Assume that $\mathbf{A}$ is SPD and $\mathbf{A} \mathbf{B}$ is symmetric.
Then, the eigenvalues of $\mathbf{B}$ are positive (negative) if and only if $\mathbf{A} \mathbf{B}$ is positive (negative) definite.
\end{lemma}
\begin{proof}
The lemma follows from Sylvester's law of inertia and $\mathbf{A} \mathbf{B} \sim \mathbf{A}^{\frac{1}{2}} \mathbf{B} \mathbf{A}^{\frac{1}{2}} \equiv \mathbf{B}$.
\end{proof}

\begin{lemma} \label{lm:CAB}
If $\mathbf{A}$ is symmetric and definite (SD), i.e., either SPD or SND, and $\mathbf{B} = \mathbf{B}^\mathsf{T}$, then $\sigma(\mathbf{A} \mathbf{B}) \subset \mathbb{R}$.
\end{lemma}
\begin{proof}
Assume that $\mathbf{A}$ is SPD.
Then, we have $\sigma( \mathbf{A}^{\frac{1}{2}} \mathbf{B} \mathbf{A}^{-\frac{1}{2}}) \subset \mathbb{R}$.
Hence, the eigenvalues of $\mathbf{A} \mathbf{B} \sim \mathbf{A}^{\frac{1}{2}} \mathbf{B} \mathbf{A}^{\frac{1}{2}} \equiv \mathbf{B}$ are real.
On the other hand, assume that $\mathbf{A}$ is SND.
Since $- \mathbf{A}$ is SPD, $\sigma[(-\mathbf{A}) (-\mathbf{B})] \subset \mathbb{R}$.
\end{proof}

Note that similar statements to Lemma~\ref{lm:ABC} and \ref{lm:CAB} hold if exchanging the roles of $A$ and $B$.

\begin{lemma}\label{lm:oddsumpos}
If $\sigma(\mathbf{A}) \subset \mathbb{R}$, then the eigenvalues of $\sum_{i = 0}^{\ell-1} \mathbf{A}^i$ are positive for all $\ell$ odd.
\end{lemma}
\begin{proof}
If $\lambda$ is an eigenvalue of $\mathbf{A}$ not equal to $1$, then the corresponding eigenvalue of $\sum_{i=0}^{\ell - 1} \mathbf{A}^i$ satisfies $(1 - \lambda^\ell) / (1 - \lambda) > 0$.
If $\lambda = 1$ is an eigenvalue of $\mathbf{A}$, then the corresponding eigenvalue of $\sum_{i=0}^{\ell - 1} \mathbf{A}^i$ satisfies $\ell > 0$.
Hence, the eigenvalues of $\sum_{i=0}^{\ell - 1} \mathbf{A}^i$ are positive for all $\ell$ odd.
\end{proof}

The following theorem gives conditions such that the inner-iteration preconditioning matrix for a symmetric matrix is SD.

\begin{theorem}[\textrm{cf.\ \cite[Theorem~1]{Adams1985}}] \label{th:definit}
Let $\mathbf{A}$ be a symmetric matrix that has the splitting $A = M - N$.
Suppose that $\ell \in \mathbb{N}$, $\mathbf{M} = \mathbf{M}^\mathsf{T}$, $\mathbf{N}$, $\mathbf{H} = \mathbf{M}^{-1} \mathbf{N}$, and  $\mathbf{C}^{(\ell)} = \sum_{i = 0}^{\ell - 1} \mathbf{H}^i \mathbf{M}^{-1}$.
Then, the following hold.
\begin{enumerate}
	\item $\mathbf{C}^{(\ell)}$ is symmetric.
	\item For $\ell$ odd, $\mathbf{C}^{(\ell)}$ is positive definite if and only if $\mathbf{M}$ is positive definite.
	\item for $\ell$ even, $\mathbf{C}^{(\ell)}$ is positive definite if and only if $\mathbf{M} + \mathbf{N}$ is positive definite.
\end{enumerate}
\end{theorem}
\begin{proof}
Since $\mathbf{N} = \mathbf{N}^\mathsf{T}$, $\mathbf{M}^{-1} \! \mathbf{N} \mathbf{M}^{-1} \! \mathbf{N} \cdots \mathbf{M}^{-1}$ is symmetric.
Hence, $\mathbf{C}^{(\ell)}$ is symmetric.

Assume that $\ell$ is odd and $\mathbf{M}$ is SPD.
Then, from Lemma~\ref{lm:CAB}, the eigenvalues of $\mathbf{H} = \mathbf{M}^{-1} (\mathbf{M} -\mathbf{A})$ are real. 
Since $\mathbf{M}^{-1}$ is SPD, noting Lemma~\ref{lm:oddsumpos}, $\mathbf{C}^{(\ell)} = (- \sum_{i = 0}^{\ell - 1} \mathbf{H}^i ) (- \mathbf{M}^{-1})$ is SPD.
On the other hand, assume that $\mathbf{C}^{(\ell)}$ is SPD.
Then, the eigenvalues of $\sum_{i = 0}^{\ell - 1} \mathbf{H}^i = \mathbf{C}^{(\ell)} \mathbf{M}$ are real from Lemma~\ref{lm:CAB}, and positive from Lemma~\ref{lm:oddsumpos}.
Hence, from Lemma~\ref{lm:ABC}, $\mathbf{M} = ( \mathbf{C}^{(\ell)} )^{-1} \sum_{i = 0}^{\ell - 1} \mathbf{H}^i = ( - \mathbf{C}^{(\ell)} )^{-1} ( - \sum_{i = 0}^{\ell - 1} \mathbf{H}^i )$ is positive definite.

Since $\ell$ is even, we have
\begin{align}
	\mathbf{M} \mathbf{C}^{(\ell)} \mathbf{M} & = \mathbf{M} + \mathbf{M} \mathbf{H} + \mathbf{M} \mathbf{H}^2 + \mathbf{M} \mathbf{H}^3 + \cdots + \mathbf{M} \mathbf{H}^{\ell-1} \notag \\
	& = ( \mathbf{M} + \mathbf{M} \mathbf{H}) + (\mathbf{M} + \mathbf{M} \mathbf{H}) \mathbf{H}^2 + (\mathbf{M} + \mathbf{M} \mathbf{H}) \mathbf{H}^4 + \cdots + (\mathbf{M} + \mathbf{M} \mathbf{H} ) \mathbf{H}^{\ell - 2} \notag \\
	& = (\mathbf{M} + \mathbf{N}) (\mathbf{I} + \mathbf{H}^2 + \mathbf{H}^4 + \cdots + \mathbf{H}^{\ell - 2}). 
	\label{eq:MCM}
\end{align}

Assume that $\mathbf{M} + \mathbf{N}$ is SPD.
Since $\mathbf{G} = \sum_{i = 0}^{(\ell - 2) / 2} (\mathbf{H}^2)^i = (\mathbf{M} + \mathbf{N})^{-1} \mathbf{M} \mathbf{C}^{(\ell)} \mathbf{M}$, we have $\sigma (\mathbf{G}) \subset \mathbb{R}$ from Lemma~\ref{lm:CAB}, and $\lambda > 0$ for all $\lambda \in \sigma (G)$ from Lemma~\ref{lm:oddsumpos}.
Hence, $\mathbf{M} \mathbf{C}^{(\ell)} \mathbf{M} = (\mathbf{M} + \mathbf{N}) \mathbf{G} = - (\mathbf{M} + \mathbf{N}) (- \mathbf{G}) \equiv \mathbf{C}^{(\ell)}$ is positive definite.

On the other hand, assume that $\mathbf{C}^{(\ell)} \equiv \mathbf{M} \mathbf{C}^{(\ell)} \mathbf{M}$ is SPD.
Then, from \eqref{eq:MCM}, $\mathbf{M} + \mathbf{N}$ is nonsingular.
Since $(\mathbf{M} + \mathbf{N})^{-1}$ is symmetric, the eigenvalues of $\mathbf{G} = (\mathbf{M} + \mathbf{N})^{-1} \mathbf{M} \mathbf{C}^{(\ell)} \mathbf{M}$ are real from Lemma~\ref{lm:CAB}, and positive from Lemma~\ref{lm:oddsumpos}.
Hence, $\mathbf{M} + \mathbf{N} = \mathbf{M} \mathbf{C}^{(\ell)} \mathbf{M} \mathbf{G}^{-1} = (- \mathbf{M} \mathbf{C}^{(\ell)} \mathbf{M}) (- \mathbf{G}^{-1})$ is positive definite from Lemma~\ref{lm:ABC}.
\end{proof}

Letting $\mathbf{A}$ be positive definite in Theorem~\ref{th:definit}, we obtain \cite[Theorem~1]{Adams1985} as a corollary.

\subsection{Convergence conditions.}
We give sufficient conditions such that CG and MINRES preconditioned by inner iterations determine a solution of symmetric linear systems.

\begin{theorem} \label{th:iCGconv}
Assume that $\mathbf{A} = \mathbf{A}^\mathsf{T}$ is not necessarily nonsingular and $\mathbf{M} = \mathbf{M}^\mathsf{T}$ is nonsingular such that $\mathbf{A} = \mathbf{M} - \mathbf{N}$.
Then, CG preconditioned by $\ell$ steps of the inner iterations $\mathbf{C}^{(\ell)}$ defined above with $\mathbf{M}$ definite for $\ell$ odd and $\mathbf{M} + \mathbf{N}$ definite for $\ell$ even, determines a solution of $\mathbf{A} \mathbf{x} = \mathbf{b}$ with $\mathbf{A}$ SPSD for all $\mathbf{b} \in \mathcal{R}(\mathbf{A})$ and for all $\mathbf{x}_0 \in \mathbb{R}^n$.
\end{theorem}
\begin{proof}
From Theorem~\ref{th:definit}, $\mathbf{C}^{(\ell)}$ is SD for all $\ell \in \mathbb{N}$.
Lemma~\ref{lm:PCG} applied to CG for $\mathbf{C}^{(\ell)} \mathbf{A} ^\mathsf{T} \mathbf{x} = \mathbf{C}^{(\ell)} \mathbf{b}$ with the ${\mathbf{C}^{(\ell)}}^{-1}$ inner product gives the theorem.
\end{proof}

\begin{theorem} \label{th:iMRconv}
Under the same assumption in Theorem~\ref{th:iCGconv}, MINRES preconditioned by $\ell$ steps of the inner iterations $\mathbf{C}^{(\ell)}$ defined above with $\mathbf{M}$ definite for $\ell$ odd and $\mathbf{M} + \mathbf{N}$ definite for $\ell$ even, determines a solution of $\mathbf{A} \mathbf{x} = \mathbf{b}$ for all $\mathbf{b} \in \mathcal{R}(\mathbf{A})$ and for all $\mathbf{x}_0 \in \mathbb{R}^n$.
\end{theorem}
\begin{proof}
From Theorem~\ref{th:definit}, $\mathbf{C}^{(\ell)}$ is SD for all $\ell \in \mathbb{N}$.
Lemma~\ref{lm:PMR} applied to MINRES for $\mathbf{C}^{(\ell)} \mathbf{A} ^\mathsf{T} \mathbf{x} = \mathbf{C}^{(\ell)} \mathbf{b}$ the ${\mathbf{C}^{(\ell)}}^{-1}$ inner product gives the theorem.
\end{proof}

The solutions determined by these methods are given similarly to \eqref{eq:PKsol} with $\mathbf{P}^{-1} = \mathbf{C}^{(\ell)}$.

Theorems~\ref{th:iCGconv} and \ref{th:iMRconv} will be applied to CG and MINRES-type methods for least squares and minimum-norm solution problems in Sections~\ref{sec:aplLS} and \ref{sec:aplminsol}.

We note a relationship among definiteness, P-regularity, and semiconvergence.
For a square matrix $\mathbf{A}$, we say the splitting $\mathbf{A} = \mathbf{M} - \mathbf{N}$ is P-regular if $\mathbf{M}$ is nonsingular and $\mathbf{M} + \mathbf{N}$ is positive definite, i.e., the symmetric part of $\mathbf{M} + \mathbf{N}$ is SPD.
Let $\mathbf{A} = \mathbf{M} - \mathbf{N}$ be P-regular for $\mathbf{A}$ symmetric, equivalently $\mathbf{M} + \mathbf{M}^\mathsf{T} - \mathbf{A}$ positive definite.
Note that if $\mathbf{M} = \mathbf{M}^\mathsf{T}$, then $\mathbf{M} + \mathbf{M}^\mathsf{T} - \mathbf{A} = 2 \mathbf{M} - \mathbf{A} = \mathbf{M} + \mathbf{N}$.
Then, $\mathbf{H} = \mathbf{M}^{-1} \mathbf{N}$ is semiconvergent, i.e., $\lim_{i \rightarrow \infty} \mathbf{H}^i$ exists, if and only if $\mathbf{A}$ is positive semidefinite \cite[Theorem~2]{Keller1965}.
Hence, for $\mathbf{A}$ indefinite, $\mathbf{H} = \mathbf{M}^{-1} \mathbf{N}$ is not semiconvergent even if $\mathbf{A} = \mathbf{M} - \mathbf{N}$ is P-regular.
Therefore, from Theorem~\ref{th:iMRconv}, MINRES preconditioned by the inner iterations can determine a solution of $\mathbf{A} \mathbf{x} = \mathbf{b}$ even if $\mathbf{H}$ is not semiconvergent, i.e., divergent.
For exmaple, if $\mathbf{A} = \diag (1, -1) = \mathbf{M} - \mathbf{N}$, $\mathbf{M} = \mathbf{I}$, and $\mathbf{N} = \diag (0, 2)$, then $\mathbf{M}$ and $\mathbf{M} + \mathbf{N}$ are SPD but $\mathbf{H} = \mathbf{M}^{-1} \mathbf{N} = \mathbf{N}$ is not semiconvergent.

\subsection{Convergence bounds.} \label{sec:bound}
Consider convergence bounds of CG and MINRES preconditioned by inner iterations.
For the definiteness of the preconditioning matrix, assume that $\mathbf{C}^{(\ell)}$ is SPD, or $\mathbf{M}$ and $\mathbf{M}+\mathbf{N}$ are definite for $\ell$ both odd and even from Theorem~\ref{th:definit}.

First, we focus on CG.
Assume that $\mathbf{A}$ is SPSD and let $\mathbf{H} = \mathbf{M}^{-1} \mathbf{N}$.
From the proof of Theorem~\ref{th:definit}, we have $\sigma(\mathbf{H}) \subset \mathbb{R}$.
Denote the pseudo spectral radius of $\mathbf{H}$ by $\nu (\mathbf{H}) = \max \lbrace |\lambda|: \lambda \in \sigma (\mathbf{H}) \backslash \lbrace 1 \rbrace \rbrace$ and the largest and smallest eigenvalues of $\mathbf{H}$ not equal to $1$ by $\lambda_{\max}(\mathbf{H})$ and $\lambda_{\min}(\mathbf{H})$, respectively.
Since $\mathbf{H}$ is semiconvergent \cite[Theorem~2]{Keller1965}, equivalently $\nu(\mathbf{H}) < 1$ and the eigenvalues of $\mathbf{H}$ equal to $1$ are simple \cite{Hensel1926}, we have 
\begin{align*}
	\kappa_2 (\mathbf{C}^{(\ell)} \mathbf{A}) = 
	\begin{cases}
		[1 - \lambda_{\max}(\mathbf{H})] / [1 - \lambda_{\min}(\mathbf{H})] & \mbox{for} \quad \ell \mbox{ odd}, \\
		(1 - \delta^\ell) / [1 - \nu(\mathbf{H})^\ell] & \mbox{for} \quad \ell \mbox{ even},
	\end{cases}	
\end{align*}
where $\delta$ is the eigenvalue with the smallest absolute value of $\mathbf{H}$.
If $\kappa^{(\ell)} = \kappa_2 (\mathbf{C}^{(\ell)} \mathbf{A})$, then an error bound of CG preconditioned by $\ell$ inner iterations is given as $\| \mathbf{e}_k \|_A \leq 2 [(\sqrt{\kappa^{(\ell)}} - 1) / (\sqrt{\kappa^{(\ell)}} + 1)]^k \| \mathbf{e}_0 \|_A$.
Thus, similar arguments in \cite[Section 2.2]{Adams1985} can be applied to the present SPSD (SNSD) case, defining the smallest eigenvalue of the iteration matrix by $\lambda_r$.
In order to avoid repetition, we omit the detail.

On the other hand, we give a bound of MINRES preconditioned by $\ell$ inner iterations.

\begin{theorem} \label{th:iMRb}
If $\mathbf{A}$ is SPSD and $\mathbf{H}$ is semiconvergent, then the $k$th residual $\mathbf{r}_k$ of MINRES preconditioned by $\ell$ steps of the inner iterations define above satisfies 
\begin{align}
	\| \check{\mathbf{r}}_k \|_2 \leq \min \left[ \nu (\mathbf{H})^{k \ell}, 2 \left( \frac{\sqrt{\kappa^{(\ell)}} - 1}{\sqrt{\kappa^{(\ell)}} + 1} \right)^k \right] \| \check{\mathbf{r}}_0 \|_2
	\label{eq:iMRb}
\end{align}
for all $\mathbf{b} \in \mathcal{R}(\mathbf{A})$ and for all $\mathbf{x}_0 \in \mathbb{R}^n$, where $\check{\mathbf{r}}_k = {\mathbf{C}^{(\ell)}}^{\frac{1}{2}} \mathbf{r}_k$.
\end{theorem}
\begin{proof}
Theorem~\ref{th:iMRconv} ensures that MINRES preconditioned by the $\ell$ steps of the inner iterations determines a solution of $\mathbf{A} \mathbf{x} = \mathbf{b}$ for all $\mathbf{b} \in \mathcal{R}(\mathbf{A})$ and for all $\mathbf{x}_0 \in \mathbb{R}^n$.
From \cite[Theorem~1]{Bai2000}, we have 
\begin{align*}
	\| \check{\mathbf{r}}_k \|_2 & = \min_{\genfrac{}{}{0pt}{2}{p \in \mathbb{P}_k}{p(0) = 1}} \left\| p(\mathbf{A} \mathbf{C}^{(\ell)}) \check{\mathbf{r}}_0 \right\|_2 \leq \left( \min_{\genfrac{}{}{0pt}{2}{p \in \mathbb{P}_k}{p(0) = 1}} \max_{\lambda \in \sigma(\mathbf{A} \mathbf{C}^{(\ell)})} | p(\lambda) | \right) \| \check{\mathbf{r}}_0 \|_2.
\end{align*}
Since the eigenvalues not equal to zero of $\mathbf{A} \mathbf{C}^{(\ell)}$ are in the circle with radius $\rho(\mathbf{H})^\ell < 1$ with center at $1$, \cite[Theorems 2, 5]{Bai2000} gives the bound $\nu(\mathbf{H})^{k \ell}$ of the first factor.
Similarly to the error bound of CG using the condition number of the coefficient matrix, the residual bound of MINRES is obtained (cf.\ \cite[Secion 6.11.3]{Saad2003}).
\end{proof}

From Theorem~\ref{th:iMRb}, the convergence of MINRES preconditioned by inner iterations is expected to be fast as the spectral radius is small and/or the number of inner iterations are large.

We compare the convergence of MINRES preconditioned by inner iterations with the stationary iterative method alone that is used as inner iterations for MINRES.
If their (inner) iteration matrices are the same and semiconvergent, then the convergence of MINRES preconditioned by inner iterations is not worse in terms of the number of outer iterations vs.\ the residual norm.
This is because the convergence factor of the stationary iterative method is $\nu(\mathbf{H})^k$, which is larger than the factors in \eqref{eq:iMRb}.
However, their computational costs of each iteration are not the same, the total costs required to attain a certain stopping criterion are easily comparable in theory.

Since it is assumed in Theorem~\ref{th:iMRb} that $\mathbf{H}$ is semiconvergent, which is weaker than that $\mathbf{M} + \mathbf{N}$ is SPD, Theorem~\ref{th:iMRb} looses generality concerning the indefiniteness.
We showed Theorem~\ref{th:iMRb} for the application of MINRES to the normal equations, whose coefficient matrices are SPSD.

\subsection{Specific inner-iteration preconditioning methods.} \label{sec:spec}
Theorem~\ref{th:definit} gives insights for justifying the use of specific stationary iterative methods as inner-iteration preconditioning for CG and MINRES for solving SPSD and indefinite systems, respectively.
Let $\omega \in \mathbb{R}$ hereafter.
The splitting matrix $\omega^{-1} \mathbf{I}$ of $\mathbf{A}$ gives the Richardson method for \eqref{eq:LS} if $\omega \not = 0$.
For odd $\ell$, the inner-iteration preconditioning matrix $\mathbf{C}^{(\ell)}$ of the Richardson method is definite if $\omega \not = 0$.
Let $\mathbf{A} = \mathbf{L} + \mathbf{D} + \mathbf{L}^\mathsf{T}$, where $\mathbf{L}$ is strictly lower triangular and $\mathbf{D}$ is diagonal.
Then, the splitting matrix $\omega^{-1} \mathbf{D}$ of $\mathbf{A}$ with $\mathbf{D}$ nonsingular gives JOR for \eqref{eq:LS} if $\omega \not = 0$.
For odd $\ell$, the inner-iteration preconditioning matrix $\mathbf{C}^{(\ell)}$ of JOR is definite if $\omega \not = 0$ and $\mathbf{D}$ is definite.
For $\ell$ even, we show the following.
Note that the splitting $\mathbf{A} = \mathbf{M} - \mathbf{N}$ gives $\mathbf{M} + \mathbf{N} = 2 \mathbf{M} - \mathbf{A}$.

\begin{lemma} \label{lm:elleven}
Let $\mathbf{A}$ be a symmetric matrix, $\mathbf{B}$ be an SPD matrix, and $\mathbf{M} = \omega^{-1} \mathbf{B}$.
Denote the largest eigenvalue of $\mathbf{B}^{-\frac{1}{2}} \mathbf{A} \mathbf{B}^{-\frac{1}{2}}$ by $\lambda_\mathrm{max}$.
Then, $2 \mathbf{M} - \mathbf{A}$ is SPD if and only if $\omega \in (0, 2 / \lambda_\mathrm{max})$ for $\lambda_\mathrm{max} > 0$, $\omega \not \in [2 / \lambda_\mathrm{max}, 0]$ for $\lambda_\mathrm{max} < 0$, or $\omega > 0$ for $\lambda_\mathrm{max} = 0$.
\end{lemma}
\begin{proof}
Denote an eigenvalue of $\mathbf{B}^{-\frac{1}{2}} \mathbf{A} \mathbf{B}^{-\frac{1}{2}}$ by $\lambda$.
Then, the corresponding eigenvalue of $2 \omega^{-1} \mathbf{I} - \mathbf{B}^{-\frac{1}{2}} \mathbf{A} \mathbf{B}^{-\frac{1}{2}} \equiv 2 \omega^{-1} \mathbf{B} - \mathbf{A} = 2 \mathbf{M} - \mathbf{A}$ is $2 \omega^{-1} - \lambda$.

Since $2 \omega^{-1} - \lambda > 0$ for all $\lambda \in \sigma (\mathbf{B}^{-\frac{1}{2}} \mathbf{A} \mathbf{B}^{-\frac{1}{2}})$ is equivalent to that $2\mathbf{M} - \mathbf{A}$ is SPD, we have the intervals of $\omega$ for the positive definiteness of $2 \mathbf{M} - \mathbf{A}$.
\end{proof}

With the splitting matrices $\mathbf{B} = \mathbf{I}$ and $\mathbf{D}$ SPD in Lemma~\ref{lm:elleven}, we obtain the interval of the relaxation parameter $\omega$ for the definiteness of the Richardson and JOR inner-iteration preconditioning matrices $\mathbf{C}^{(\ell)}$ for $\ell$ even.
We omit the details to avoid redundancy.
From Lemma~\ref{lm:elleven}, with an SPD splitting matrix not necessarily diagonal, we can generalized JOR.

Next, consider the inner-iteration preconditioning using SSOR for \eqref{eq:LS}.
Let $\omega^{-1} \mathbf{D} + \mathbf{L}$ be the splitting matrix of $\mathbf{A}$ for the forward sweep of SSOR and $\omega^{-1} \mathbf{D} + \mathbf{L}^\mathsf{T}$ be that of $\mathbf{A}$ for the backward sweep.
\begin{theorem}
Assume that $\mathbf{D}$ is SPD.
Then, the SSOR splitting matrix $\mathbf{M} = \omega^{-1} (2 - \omega)^{-1} (\mathbf{D} + \omega \mathbf{L}) \mathbf{D}^{-1} (\mathbf{D} + \omega \mathbf{L}^\mathsf{T})$ of $\mathbf{A}$ is nonsingular if and only if $\omega \not = 0$, $2$.
For $\ell$ odd, the SSOR inner-iteration preconditioning matrix is SPD if and only if $\omega \in (0, 2)$.
Let $\mu = \lambda_\mathrm{min} (2 \mathbf{D}^{-\frac{1}{2}} \mathbf{L} \mathbf{D}^{-1} \mathbf{L}^\mathsf{T} \mathbf{D}^{-\frac{1}{2}} + \mathbf{D}^{-\frac{1}{2}} (\mathbf{L}+\mathbf{L}^\mathsf{T}) \mathbf{D}^{-\frac{1}{2}}) + 1$.
Then, for $\ell$ even, if $\omega$ satisfies
\begin{align}
			\omega& < \frac{1 + \sqrt{1-2\mu}}{\mu}, ~ 0 < \omega < \frac{1-\sqrt{1-2\mu}}{\mu},~ 2 < \omega, &&& \quad \mu &< 0, \label{eq:intval1} \\
			0 < \omega& < \frac{1-\sqrt{1-2\mu}}{\mu}, ~ 2 < \omega < \frac{1+\sqrt{1-2\mu}}{\mu},&&& \quad 0 < \mu &< \frac{1}{2}, \label{eq:intval2}\\
			& 0 < \omega < 2, &&& \quad \frac{1}{2} < \mu,& \label{eq:intval3}
\end{align}
then the SSOR inner-iteration preconditioning matrix is SPD 
\end{theorem}
\begin{proof}
Since $\omega^{-1} (2 - \omega)^{-1} (\mathbf{D} + \omega \mathbf{L}) \mathbf{D}^{-1} (\mathbf{D} + \omega \mathbf{L}^\mathsf{T}) \equiv \omega^{-1} (2 - \omega)^{-1} \mathbf{I}$, the SSOR splitting matrix of $\mathbf{A}$ is nonsingular if $\omega \not = 0, 2$.
For $\ell$ odd, from Theorem~\ref{th:definit}, the SSOR inner-iteration preconditioning matrix is SPD if and only if $\omega \in (0, 2)$.
Next, let $\ell$ even.
Assume $\omega \not = 0$, $2$.
Noting 
\begin{align*}
	2 \mathbf{M} - \mathbf{A} & = 2 \omega^{-1} (2 - \omega)^{-1} (\mathbf{D} + \omega \mathbf{L}) \mathbf{D}^{-1} (\mathbf{D} + \omega \mathbf{L}^\mathsf{T}) - \mathbf{A} \\
	& \equiv \omega^{-1} (2-\omega)^{-1} \lbrace \omega^2 [2 \mathbf{D}^{-\frac{1}{2}} \mathbf{L} \mathbf{D}^{-1} \mathbf{L}^\mathsf{T} \mathbf{D}^{-\frac{1}{2}} +  \mathbf{D}^{-\frac{1}{2}} (\mathbf{L} + \mathbf{L}^\mathsf{T}) \mathbf{D}^{-\frac{1}{2}} + \mathbf{I}] -2 (\omega - 1) \mathbf{I} \rbrace,
\end{align*}
let 
\begin{align*}
	\mathbf{G}(\omega) & = \omega (2 - \omega) \mathbf{D}^{-\frac{1}{2}} (2 \mathbf{M} - \mathbf{A}) \mathbf{D}^{-\frac{1}{2}} \\
	& = \omega^2 [2 \mathbf{D}^{-\frac{1}{2}} \mathbf{L} \mathbf{D}^{-1} \mathbf{L}^\mathsf{T} \mathbf{D}^{-\frac{1}{2}} +  \mathbf{D}^{-\frac{1}{2}} (\mathbf{L} + \mathbf{L}^\mathsf{T}) \mathbf{D}^{-\frac{1}{2}} + \mathbf{I}] -2 (\omega - 1) \mathbf{I} 
\end{align*}

Since $\lambda \geq \mu \omega^2 - 2 \omega + 2$ holds for all $\lambda \in \sigma(\mathbf{G}(\omega))$, for $\omega$ satisfying \eqref{eq:intval1}--\eqref{eq:intval3}, the SSOR inner-iteration precondiotining matrix is SPD.
\end{proof}

Note that the SSOR iteration matrix of $\mathbf{A}$ SPSD with $\mathbf{D}$ SPD and $\omega \in (0, 2)$ is semiconvergent \cite[Theorem~14]{Dax1990}.

\section{Application to least squares problems.} \label{sec:aplLS}
Consider solving linear least squares problems \eqref{eq:LSprob}.
We give the convergence theory of CGLS, LSQR, and LSMR preconditioned by inner iterations \cite{MorikuniHayami2013, MorikuniHayami2015} with setting $\mathbf{A} = A^\mathsf{T} \! A$ and $\mathbf{b} = A^\mathsf{T} \boldsymbol{b}$ in this section, by applying results in Section~\ref{sec:symmLS}. 

We form the preconditioned matrix for CGLS and LSMR with $\ell$ inner iterations, similarly to Section~\ref{sec:iCG}.
In each iteration of CGLS and LSMR, we apply $\ell$ steps of a stationary iterative method to $A^\mathsf{T} \! A \boldsymbol{z} = \boldsymbol{s}_k$ of the normal equations of the second kind (lines 2 and 6 in \cite[Algorithm~E.1]{MorikuniHayami2015} and lines 2 and 6 in \cite[Algorithm E.2]{MorikuniHayami2015}).
Here, $\boldsymbol{s}_k$ is the vector depending on the number of iterations $k$.

Let $M$ be a nonsingular matrix such that $A^\mathsf{T} \! A = M - N$.
Denote the iteration matrix by $H = M^{-1} N$.
Assume that the initial iterate is $\boldsymbol{z}^{(0)} = \boldsymbol{0}$.
Then, the $\ell$th iterate of the stationary iterative method is $\boldsymbol{z}^{(\ell)} = H \boldsymbol{z}^{(\ell - 1)} + M^{-1} \boldsymbol{s}_k = \sum_{i = 0}^{\ell - 1} H^i M^{-1} \boldsymbol{s}_k$, $\ell > 0$.
Hence, the preconditioning matrix is $C^{(\ell)} = \sum_{i = 0}^{\ell - 1} H^i M^{-1}$.
Therefore, the preconditioned matrix is $C^{(\ell)} A^\mathsf{T} \! A = \sum_{i = 0}^{\ell - 1} H^i (\I - H) = \I - H^\ell$.
See \cite{Santos2003} for a different formulation of CGLS preconditioned by the SSOR splitting.

Now we give conditions such that CGLS, LSQR, and LSMR preconditioned by inner iterations determine a least squares solution.

\begin{theorem} \label{th:iCGLSLSMR}
Let $A \in \mathbb{R}^{m \times n}$.
Assume that $M = M^\mathsf{T}$ is a nonsingular matrix such that $A^\mathsf{T} \! A = M - N$.
Then, CGLS, LSQR, and LSMR preconditioned by $\ell$ steps of the inner iterations defined above with $M$ definite for $\ell$ odd and $M + N$ definite for $\ell$ even, respectively, determine a solution of $\min_{\boldsymbol{x} \in \mathbb{R}^n} \| \boldsymbol{b} - A \boldsymbol{x} \|_2$ for all $\boldsymbol{b} \in \mathbb{R}^m$ and for all $\boldsymbol{x}_0 \in \mathbb{R}^n$.
\end{theorem}
\begin{proof}
Since $A^\mathsf{T} \! A$ is SPSD, the following hold from Theorem~\ref{th:definit}.
For $\ell$ odd, $C^{(\ell)}$ is SPD if and only if $M$ is SPD.
For $\ell$ even, $C^{(\ell)}$ is SPD if and only if $M + N$ is SPD.
Hence, Theorems~\ref{th:iCGconv} and \ref{th:iMRconv} complete the proof.
\end{proof}

\begin{remark} \label{rem:1}
This theorem holds whether $A$ is of full-rank or rank-deficient, and whether $A$ is overdetermined or underdetermined, i.e., unconditionally with respect to $A$.
\end{remark}

\begin{remark} \label{rem:2}
We can derive bounds of these methods under the conditions of Theorem~\ref{th:iCGLSLSMR} from Section~\ref{sec:bound}.
\end{remark}

There are efficient implementations of inner-iteration preconditioning without explicitly forming $A A^\mathsf{T}$ such as the Richardson-NE, Cimmino-NE, and NE-SSOR methods \cite[Appendix D]{MorikuniHayami2015}, which are mathematically equivalent to the Richardson method, JOR, and SSOR applied to the normal equations of the second kind, respectively.
CGLS preconditioned by one step of NE-SSOR was considered in \cite{BjorckElfving1979}.
Let $A^\mathsf{T} \! A = L + D + L^\mathsf{T}$, where $L$ is strictly lower triangular and $D$ is diagonal.
Assume that $A$ has no zero columns.
Then, $D$ is SPD.
If $\mathbf{A} = A^\mathsf{T} \! A$ and $\mathbf{b} = A^\mathsf{T} \boldsymbol{b}$ in Section~\ref{sec:spec}, then we obtain the intervals of the parameter values of Richardson-NE, Cimmino-NE, and NE-SSOR such that their inner-iteration preconditioning matrices are SD.
Thus, from Theorem~\ref{th:iCGLSLSMR}, CGLS, LSQR, and LSMR preconditioned by these inner iterations with relaxation parameters within the intervals determines a solution of $\min_{\boldsymbol{x} \in \mathbb{R}^n} \| \boldsymbol{b} - A \boldsymbol{x} \|_2$ for all $\boldsymbol{b} \in \mathbb{R}^m$ and for all $\boldsymbol{x}_0 \in \mathbb{R}^n$.

\section{Application to minimum-norm solution problems.}\label{sec:aplminsol}
Consider solving minimum-norm solution problems \eqref{eq:minsolLS}.
Results in Section~\ref{sec:symmLS} can be applied to CGNE and MRNE preconditioned by inner iterations with $\mathbf{A} = A A^\mathsf{T}$ and $\mathbf{b} = \boldsymbol{b}$ \cite{MorikuniHayami2013, MorikuniHayami2015}.

In each iteration of CGNE and MRNE, we apply $\ell$ steps of a stationary iterative method to $A A^\mathsf{T} \boldsymbol{z} = \boldsymbol{r}_k$ of the normal equations of the second kind (lines 2 and 6 in \cite[Algorithm~E.3]{MorikuniHayami2015} and lines 2 and 6 in \cite[Algorithm E.4]{MorikuniHayami2015}).
Here, $\boldsymbol{r}_k$ is the vector depending on the number of iterations $k$.
Since $A A^\mathsf{T}$ is SPSD, replacing $A^\mathsf{T} A$ by $A A^\mathsf{T}$ in the discussion of Section~\ref{sec:aplLS}, the following hold from Theorem~\ref{th:definit}.
For $\ell$ odd, $C^{(\ell)}$ is SPD if and only if $M$ is SPD.
For $\ell$ even, $C^{(\ell)}$ is SPD if and only if $M + N$ is SPD.
Thus, we obtain the following similarly to Theorem~\ref{th:iCGLSLSMR}.



\begin{theorem} \label{th:iCGNEMRNE}
Let $A \in \mathbb{R}^{m \times n}$ and $M = M^\mathsf{T}$ be nonsingular such that $A A^\mathsf{T} = M - N$.
Then, the CGNE and MRNE methods preconditioned by $\ell$ steps of the inner iterations defined above with $M$ definite for $\ell$ odd and $M + N$ definite for $\ell$ even, respectively, determine the minimum-norm solution of $A \boldsymbol{x} = \boldsymbol{b}$ for all $\boldsymbol{b} \in \mathcal{R}(A)$ and for all $\boldsymbol{x}_0 \in \mathcal{R}(A)$.
\end{theorem}


Similar arguments to Remarks~\ref{rem:1} and \ref{rem:2} hold under the conditions of Theorem~\ref{th:iCGNEMRNE}.

We focus on using specific stationary iterative methods for inner-iteration preconditioning.
Let $A A^\mathsf{T} = L + D + L^\mathsf{T}$, where $L$ is strictly lower triangular and $D$ is diagonal.
Assume that $A$ has no zero rows.
Then, $D$ is SPD.
If $\mathbf{A} = A A^\mathsf{T}$ and $\mathbf{b} = \boldsymbol{b}$ in Section~\ref{sec:spec}, then we obtain the intervals of the parameter values of Richardson-NE, Cimmino-NE, and NE-SSOR such that their inner-iteration preconditioning matrices are SPD.
Thus, from Theorem~\ref{th:iCGNEMRNE}, CGNE and MRNE preconditioned by these inner iterations with  relaxation parameters within the intervals determines the minimum-norm solution of $A \boldsymbol{x} = \boldsymbol{b}$ for all $\boldsymbol{b} \in \mathcal{R}(A)$ and for all $\boldsymbol{x}_0 \in \mathcal{R}(A^\mathsf{T})$.
CGNE preconditioned by one step of NE-SSOR was considered in \cite{BjorckElfving1979}.
We can generalize this to a multistep version of NE-SSOR.

\section{Conclusions.} \label{sec:conc}
We considered applying stationary iterative methods with a symmetric splitting matrix as inner-iteration preconditioning to Krylov subspace methods.
We gave conditions such that the inner-iteration preconditioning matrix is definite, and show that CG and MINRES preconditioned by the inner iterations determines a solution of symmetric linear systems including the singular case.
Applying these results to CGLS, LSQR, LSMR, CGNE, and MRNE preconditioned by inner iterations, and we guaranteed using these methods for solving least squares and minimum-norm solution problems whose coefficient matrices are not necessarily of full rank.

\section*{Acknowledgement}
The author would like to thank Professor Ken Hayami and Doctor Miroslav Rozlo\v{z}n\'{i}k for their valuable comments.
This work was supported in part by JSPS KAKENHI Grant Number~16K17639.

\bibliographystyle{siam}
\bibliography{ref}

\begin{thebibliography}{10}

\bibitem{Adams1985}
\textsc{ L.~Adams}, {\em $m$-step preconditioned conjugate gradient methods},
  SIAM J. Sci. Stat. Comput., 6 (1985), pp.~452--463.

\bibitem{Axelsson1976}
\textsc{ O.~Axelsson}, {\em A class of iterative methods fo finite elements
  equations}, Comput. Methods Appl. Mech. Engrg., 9 (1976), pp.~123--137.

\bibitem{Bai2000}
\textsc{ Z.-Z. Bai}, {\em Sharp error bounds of some {K}rylov subspace methods
  for non-{H}ermitian linear systems}, Appl. Math. Comput., 109 (2000),
  pp.~273--285.

\bibitem{BjorckElfving1979}
\textsc{ {\AA}.~Bj{\"o}rck and T.~Elfving}, {\em Accelerated projection methods
  for computing pseudoinverse solutions of systems of linear equations}, BIT,
  19 (1979), pp.~145--163.

\bibitem{BrownWalker1997}
\textsc{ P.~N. Brown and H.~F. Walker}, {\em {GMRES} on (nearly) singular
  systems}, SIAM J. Matrix Anal. Appl., 18 (1997), pp.~37--51.

\bibitem{CerdanMarinMartinez2002}
\textsc{ J.~Cerd\'{a}n, J.~Mar\'{i}n, and A.~Mart\'{i}nez}, {\em Polynomial
  preconditioners based on factorized sparse approximate inverses}, Appl. Math.
  Comput., 133 (2002), pp.~171--186.

\bibitem{Cesari1937}
\textsc{ L.~Cesari}, {\em Sulla risoluzione dei sistemi di equazioni lineari
  per approssimazioni successive}, Atti Accad. Nazionale Lincei R. Classe Sci.
  Fis. Mat. Nat., 25 (1937), pp.~422--429 (in Italy).

\bibitem{Craig1955}
\textsc{ E.~J. Craig}, {\em The ${N}$-step iteration procedures}, J. Math. and
  Phys., 34 (1955), pp.~64--73.

\bibitem{Dax1990}
\textsc{ A.~Dax}, {\em The convergence of linear stationary iterative processes
  for solving singular unstructured systems of linear equations}, SIAM Rev., 32
  (1990), pp.~611--635.

\bibitem{DuboisGreenbaumRodrigue1979}
\textsc{ P.~F. Dubois, A.~Greenbaum, and G.~H. Rodrigue}, {\em Approximating
  the inverse of a matrix for use in iterative algorithms on vector
  processors}, Computing, 22 (1979), pp.~257--268.

\bibitem{EisenstatElmanSchultz1983}
\textsc{ S.~C. Eisenstat, H.~C. Elman, and M.~H. Schultz}, {\em Variational
  iterative methods for nonsymmetric systems of linear equations}, SIAM J.
  Numer. Anal., 20 (1983), pp.~345--357.

\bibitem{EisenstatOrtegaVaughan1990}
\textsc{ S.~C. Eisenstat, J.~M. Ortega, and C.~T. Vaughan}, {\em Efficient
  polynomial preconditioning for the conjugate gradient method}, SIAM J. Sci.
  Stat. Comput., 11 (1990), pp.~859--872.

\bibitem{FaberJoubertKnillManteuffel1996}
\textsc{ J.~Faber, W.~Joubert, E.~Knill, and T.~Manteuffel}, {\em Minimal
  residual method stronger than polynomial preconditioning}, SIAM J. Matrix
  Anal. Appl., 17 (1996), pp.~707--729.

\bibitem{FongSaunders2011}
\textsc{ D.~C.-L. Fong and M.~A. Saunders}, {\em {LSMR}: {A}n iterative
  algorithm for sparse least-squares problems}, SIAM J. Sci. Comput., 33
  (2011), pp.~2950--2971.

\bibitem{FoxHuskeyWilkinson1948}
\textsc{ L.~Fox, H.~D. Huskey, and J.~H. Wilkinson}, {\em Notes on the solution
  of algebraic linear simultaneous equations}, Q. J. Mechanics Appl. Math., 1
  (1948), pp.~149--173.

\bibitem{HayamiSugihara2011}
\textsc{ K.~Hayami and M.~Sugihara}, {\em A geometric view of {K}rylov subspace
  methods on singular systems}, Numer. Linear Algebra Appl., 18 (2011),
  pp.~449--469.

\bibitem{HayamiYin2008}
\textsc{ K.~Hayami and J.-F. Yin}, {\em On the convergece of {K}rylov subspace
  methods for rank-deficient least squares problems}, The {T}hird
  {I}nternational {C}onference on {S}cientific {C}omputing and {P}artial
  {D}ifferential {E}quations, Hong Kong, slides (2008).

\bibitem{HayamiYinIto2010}
\textsc{ K.~Hayami, J.-F. Yin, and T.~Ito}, {\em {GMRES} methods for least
  squares problems}, SIAM J. Matrix Anal. Appl., 31 (2010), pp.~2400--2430.

\bibitem{Hensel1926}
\textsc{ K.~Hensel}, {\em \"{U}ber {P}otenzreihen von {M}atrizen}, J. Reine
  Angew. Math., 155 (1926), pp.~107--110 (in German).

\bibitem{HestenesStiefel1952}
\textsc{ M.~R. Hestenes and E.~Stiefel}, {\em Methods of conjugate gradients
  for solving linear systems}, J. Research Nat. Bur. Standards, 49 (1952),
  pp.~409--436.

\bibitem{Jacobi1845}
\textsc{ C.~G.~J. Jacobi}, {\em \"{U}ber eine neue {A}ufl\"{o}sungsart der bei
  der {M}ethode der kleinsten {Q}uadrate vorkommenden linearen {G}leichungen},
  Astr. Nachr., 22 (1845), pp.~297--306.

\bibitem{Kaasschieter1988}
\textsc{ E.~F. Kaasschieter}, {\em Preconditioned conjugate gradients for
  solving singular systems}, J. Comput. Appl. Math., 24 (1988), pp.~265--275.

\bibitem{KammererNashed1972SINUM}
\textsc{ W.~J. Kammerer and M.~Z. Nashed}, {\em On the convergece of the
  conjugate gradient method for singular linear operator equations}, SIAM J.
  Numer. Anal., 9 (1972), pp.~165--181.

\bibitem{Kaniel1966}
\textsc{ S.~Kaniel}, {\em Estimates for some computational techniques in linear
  algebra}, Math. Comp., 20 (1966), pp.~369--378.

\bibitem{Keller1965}
\textsc{ H.~B. Keller}, {\em On the solution of singular and semidefinite
  linear systems by iteration}, J. Soc. Indust. Appl. Math. Ser. B Numer.
  Anal., 2 (1965), pp.~281--290.

\bibitem{MorikuniHayami2013}
\textsc{ K.~Morikuni and K.~Hayami}, {\em Inner-iteration {K}rylov subspace
  methods for least squares problems}, SIAM J. Matrix Anal. Appl., 34 (2013),
  pp.~1--22.

\bibitem{MorikuniHayami2015}
\leavevmode\vrule height 2pt depth -1.6pt width 23pt, {\em Convergence of
  inner-iteration {GMRES} methods for rank-deficient least squares problems},
  SIAM J. Matrix Anal. Appl., 36 (2015), pp.~225--250.

\bibitem{PaigeSaunders1975}
\textsc{ C.~C. Paige and M.~A. Saunders}, {\em Solution of sparse indefinite
  systems of linear equations}, SIAM J. Numer. Anal., 12 (1975), pp.~617--629.

\bibitem{PaigeSaunders1982a}
\leavevmode\vrule height 2pt depth -1.6pt width 23pt, {\em {LSQR}: {A}n
  algorithm for sparse linear equations and sparse least squares}, ACM Trans.
  Math. Software, 8 (1982), pp.~43--71.

\bibitem{Richardson1911}
\textsc{ L.~F. Richardson}, {\em The approximate arithmetical solution by
  finite differences of physical problems involving differential equations,
  with an application to the stresses in a masonry dam}, Philos. Trans. Roy.
  Soc. London, 210 (1911), pp.~307--357.

\bibitem{Saad2003}
\textsc{ Y.~Saad}, {\em {Iterative Methods for Sparse Linear Systems}}, SIAM,
  Philadelphia, PA, 2nd~ed., 2003.

\bibitem{Santos2003}
\textsc{ R.~J. Santos}, {\em Preconditioning conjugate gradient with symmetric
  algebraic reconstruction technique ({ART}) in computerized tomography}, Appl.
  Numer. Math., 47 (2003), pp.~255--263.

\bibitem{Sheldon1955}
\textsc{ J.~W. Sheldon}, {\em On the numerical solution of elliptic difference
  equations}, Math. Tables Other Aids Comput., 9 (1955), pp.~101--112.

\bibitem{Stiefel1955}
\textsc{ E.~Stiefel}, {\em Relaxationsmethoden bester {S}trategie zur
  {L}\"{o}sung linearer {G}leichungssysteme}, Comment. Math. Helv., 29 (1955),
  pp.~157--179.

\bibitem{SugiharaHayami2016TJSIAM}
\textsc{ K.~Sugihara and K.~Hayami}, {\em Right preconditioned {MINRES} using
  {E}isenstat {SSOR} for positive semidefinite systems}, Trans. JSIAM, 26
  (2016), pp.~124--166.

\bibitem{Vinsome1976}
\textsc{ P.~K.~W. Vinsome}, {\em Orthomin, an iterative method for solving
  sparse sets of simultaneous linear equations}, Proc. Fourth Symposium on
  Reservoir Simulation, Society of Petroleum Engineers of AIME,  (1976),
  pp.~149--159.

\bibitem{WeiWu2000}
\textsc{ Y.-M. Wei and H.~Wu}, {\em Convergence properties of {K}rylov subspace
  methods for singular linear systems with arbitrary index}, J. Comput. Appl.
  Math., 114 (2000), pp.~305--318.

\bibitem{YoungJea1980}
\textsc{ D.~M. Young and K.~C. Jea}, {\em Generalized conjugate-gradient
  acceleration of nonsymmetrlzable iteratlve methods}, Linear Algebra Appl., 34
  (1980), pp.~159--194.

\end{thebibliography}
\end{document}